\documentclass[12 pt]{amsart}

\usepackage{graphicx}
\usepackage{times}

\newtheorem{thm}{Theorem}[section]
\newtheorem{cor}[thm]{Corollary}

\newtheorem{defin}[thm]{Definition}
\newtheorem{lemma}[thm]{Lemma}

\newtheorem{rmk}[thm]{Remark}

\newcommand{\Sss}{\mbox{$\Sigma$}}
\newcommand{\Sssp}{\mbox{$\Sigma_{t_+}$}}
\newcommand{\Sssm}{\mbox{$\Sigma_{t_-}$}}
\newcommand{\Sssz}{\mbox{$\Sigma_{t_0}$}}

\newcommand{\Ggg}{\mbox{$\Gamma$}}

\newcommand{\bdd}{\mbox{$\partial$}}

\newcommand{\fp}{\mbox{$(f_1\uparrow \uparrow f_2)$}}
\newcommand{\fm}{\mbox{$(f_1\uparrow \downarrow f_2)$}}

\begin{document}



\title{Flipping bridge surfaces and bounds on the stable bridge number}

\author{Jesse Johnson}
\address{\hskip-\parindent
Jesse Johnson\\
  Mathematics Department \\
Oklahoma State University.}
\email{jjohnson@math.okstate.edu}

\author{Maggy Tomova}
\address{\hskip-\parindent
 Maggy Tomova\\
  Mathematics Department \\
University of Iowa.}
\email{mtomova@math.uiowa.edu}

\begin{abstract}
We show that if $K$ is a knot in $S^3$ and $\Sigma$ is a bridge sphere for $K$ with high distance and $2n$ punctures, the number of perturbations of $K$ required to interchange the two balls bounded by $\Sigma$ via an isotopy is $n$. We also construct a knot with two different bridge spheres with $2n$ and $2n-1$ bridges respectively for which any common perturbation has at least $3n-1$ bridges. We generalize both of these results to bridge surfaces for knots in any 3-manifold. 
\end{abstract}

\thanks{Research partially supported by an NSF grant.}
\maketitle

\section{Introduction}

Reidemeister~\cite{Re} and Singer~\cite{Sin} showed that any two Heegaard splittings for a 3-manifold $M$ have a common stabilization, i.e., if $\Sss$ and $\Sss'$ are two Heegaard surfaces for $M$ there exists a Heegaard surface $\Sss''$ that is isotopic to a stabilization of $\Sss$ as well as to a stabilization of $\Sss'$. A long standing question in Heegaard splittings asks what is the minimal genus of $\Sss''$ in terms of the genera of $\Sss$ and $\Sss'$. Examples of Heegaard splittings that required many stabilizations were presented in~\cite{DB},~\cite{JJ2} and~\cite{HTT}.

Bridge splittings are the natural extension of Heegaard splittings in the context of a compact orientable manifold $M$ containing a properly embedded tangle $T$. A \textit{bridge splitting} for $(M,T)$ is a triple $(\Sss, (H^+, \tau^+), (H^-, \tau^-))$ where $\Sss$ is a connected surface that decomposes $M$ into compression bodies $H^+$ and $H^-$ and decomposes $T$ into collections of arcs $\tau^+$ and $\tau^-$ that are embedded in the corresponding compression bodies in specific ways. The surface $\Sss$ is called a \textit{bridge surface} for $(M, T)$. Note that if $T=\emptyset$, then $(\Sss, (H^+, \tau^+), (H^-, \tau^-))$ is a Heegaard splitting for $M$. 

Given a bridge surface $\Sss$ of $(M,T)$ one can always obtain another bridge surface $\Sss''$ by performing stabilizations and perturbations. These operations are discussed in detail in \cite{STo3} and they behave in a manner similar to stabilizations of Heegaard splittings. In this paper we consider pairs of bridge splittings $\Sss$ and $\Sss'$ for $(M,T)$ and study bridge splittings $\Sss''$ that can be obtained from both $\Sss$ and $\Sss'$ via stabilizations and perturbations. The results we obtain are similar but somewhat weaker than the results obtained by Johnson for Heegaard splittings in \cite{JJ} and \cite{JJ2} due to the additional difficulties introduced by the presence of the knot. 

At first we will distinguish a bridge splitting $(\Sss, (H^+, \tau^+), (H^-, \tau^-))$ from the bridge splitting $(\Sss, (H^-, \tau^-), (H^+, \tau^+))$ in which the order of the compression bodies is reversed. We ask what is the minimum value of $2-\chi(\Sss'')$ such that $(\Sss'', (H''^+, \tau''^+), (H''^-, \tau''^-))$ is isotopic to stabilizations and perturbations of both bridge splittings $(\Sss, (H^+, \tau^+), (H^-, \tau^-))$ and $(\Sss, (H^-, \tau^-), (H^+, \tau^+))$. This value is  called the {\em flip Euler characteristic of} $\Sss$ and it is analogous to the flip genus of a Heegaard splitting defined in \cite{JJ}. We give a bound on this quantity in terms of the Euler characteristic of $\Sss$ and the distance of $T$ with respect to $\Sss$ (Definition \ref{def:distance}). 

\begin{thm}\label{thm:flip}
Let $T$ be a prime tangle properly embedded in a compact orientable irreducible 3-manifold $M$ and let $(\Sss, (H^+, \tau^+), (H^-, \tau^-))$ be a bridge splitting for $(M,T)$ such that $\chi(\Sss)\leq -4$. Then the flip Euler characteristic of $\Sss$ is at least $max\{2-2\chi(\Sss), d(\Sss,T)\}$. 
\end{thm}

\begin{cor}\label{cor:flip}
Let $T$ be a prime knot in $S^3$ and let $\Sss$ be a bridge sphere for $T$ with $n\geq 3$ bridges such that $d(T, \Sss) \geq 4n$. If $(\Sss'', (H''^+, \tau''^+), (H''^-, \tau''^-))$ is a minimal bridge number perturbation of both bridge splittings $(\Sss, (H^+, \tau^+), (H^-, \tau^-))$ and $(\Sss, (H^-, \tau^-), (H^+, \tau^+))$, then $T$ has exactly $2n$ bridges with respect to $\Sss''$. 
\end{cor}

We next consider the problem of distinguishing bridge surfaces without keeping track of the order of compression bodies.  To make this clear, we will consider only the bridge surface rather than the bridge splitting. In this case we obtain the following result.

\begin{thm} \label{thm:knot}
There exist infinitely many manifolds $M_{\alpha}$ each containing a knot $K_{\alpha}$ so that each pair $(M_{\alpha}, K_{\alpha})$ has two bridge surfaces $\Sss$ and $\Sss'$ with $\chi(\Sss)=2s$ and $\chi(\Sss')=2s-2$ so that for every bridge surface $\Sss''$ that is isotopic to stabilizations and perturbations of both $\Sss$ and $\Sss'$, $\chi(\Sss')\leq 3s+2$. 
\end{thm}

As a corollary of the above we obtained the following result:

\begin{cor}\label{cor:secondmain}
For every $n \geq 2$ there exists a knot $\tilde K$ in $S^3$ with bridge spheres $\Sigma$ and $\Sigma'$ with bridge numbers $2n-1$ and $2n$ respectively such that every bridge surface $\Sigma''$ which is isotopic to a perturbation of both has at least $3n-1$ bridges.
\end{cor}

In Section \ref{sec:prel} we give the definition of a bridge splitting for a pair $(M,T)$ and explain how a sweep-out is associated to any bridge splitting. Furthermore we define two conditions on a pair of sweep-outs:  A sweep-out $g$ can split a second sweep-out $f$ for the same manifold or can span it. Generically these are the only two options for how $g$ behaves with respect to $f$. 

In Sections \ref{sec:span} and \ref{sec:split} we consider two bridge splittings $\Sss$ and $\Sss'$ for $(M,T)$ with associated sweep-outs $f$ and $g$. We show that if $g$ spans $f$, then the Euler characteristic of the punctured bridge surface $\Sss$ is bounded below by the Euler characteristic of the punctured bridge surface $\Sss'$. Next we define the distance of a bridge splitting and we show that if $g$ splits $f$ then the distance of $\Sss$ is bounded above by the Euler characteristic of $\Sss'$. Finally we consider the case where $g$ neither spans not splits $f$ and we show that this can only occur if $\chi(\Sss)\geq -3$. Using these results in Section \ref{sec:flipping} we prove Theorem \ref{thm:flip} and in Section \ref{sec:example} we prove Theorem \ref{thm:knot}.

\section{Preliminaries}\label{sec:prel}

\subsection{Compression bodies containing trivial arcs} Let $H$ be a compression body. Recall that a spine of $H$ is a complex $\bdd_- H \cup \Gamma$ where $\Gamma \subset H$ is a properly embedded finite graph with no valence $1$ vertices in the interior of $H$ and such that $H$ is isotopic to a regular neighborhood of $\bdd_- H \cup \Gamma$. 
A set of properly embedded arcs $\tau=\{t_1,..., t_n\}$ in $H$ is {\em trivial} if each $t_i$ is either parallel to $\bdd_+H$ or is a vertical arc with one endpoint in $\bdd_+H$ and the other endpoint in $\bdd_-H$. If an arc is parallel to $\bdd_+H$ the disk of parallelism is called a {\em bridge disk}. We will denote the pair of a compression body $H$ containing properly embedded trivial arcs $\tau$ by $(H, \tau)$. The arcs $\tau$ can be isotoped in $H$ so that the projection $H - spine(H) \cong \bdd H \times [0, 1) \to
[0, 1)$ has no critical points in the vertical arcs and a single critical point, say a maximum, in each boundary parallel arc. Let $s_i$ be a collection of vertical arcs each connecting a single maximum of $\tau$ to a spine of $H$. Let $spine((H, \tau))=spine(H) \cup \{s_i\}$ and note that there is a map $(\bdd H, \bdd H \cap \tau) \times I \to (H, \tau)$
which is a homeomorphism except over the spine, and the map
gives a neighborhood of the spine a mapping cylinder
structure.

\subsection{Bridge splittings} Let $T$ be a properly embedded tangle in a compact oriented $3$-manifold $M$ and let $\Sss$ be a properly embedded surface transverse to $T$ such that $\Sss$ splits $M$ into two compression bodies $H^+$ and $H^-$ and such that $\tau^+=H^+\cap T$ and $\tau^-=H^-\cap T$ are trivial arcs in the corresponding compression body. In this case we say that $(\Sss, (H^+, \tau^+), (H^-, \tau^-))$ is a bridge splitting for $(M,T)$ and $\Sss$ is a bridge surface. As every compact orientable 3-manifold has a Heegaard splitting it is easy to see that every properly embedded tangle in any 3-manifold has a bridge splitting.

\subsection{Surfaces in $(M,T)$} Suppose $M$ is a compact, irreducible, orientable 3-manifold containing a properly embedded tangle $T$ and let $F$ be a surface in $M$ transverse to $T$. The surface $F$ gives rise to a punctured surface in the complement of a regular neighborhood $\eta(T)$ of $T$. We will refer to this punctured surface as $F$ also and we will specify if we are referring to the punctured or the closed surface whenever it is not clear from context. Two surfaces in $(M,T)$ will be considered isotopic only if there is an isotopy between them transverse to the tangle. 

A simple closed curve in $F-\eta(T)$ is {\em essential} if it does not bound a disk in $F$ and it is not parallel to the boundary of a puncture. A properly embedded arc in $F$ with endpoints in $F \cap \bdd M$ is essential if it does not cobound a disk with an arc in $F \cap \bdd M$. An embedded disk $D$ in $M$ is a {\em compressing disk} for $F$ if $D \cap T =\emptyset$, $D \cap F = \bdd D$ and $\bdd D$ is an essential curve in $F-\eta(T)$. A properly embedded disk $D^c$ in $M$ is a {\em cut-disk} for $F$ if $D^c \cap T$ is a single point in the interior of $D^c$, $D^c \cap F = \bdd D^c$ and $\bdd D^c$ is an essential curve in $F-\eta(T)$. A {\em c-disk} is either a cut or a compressing disk.

\subsection{Obtaining new bridge splittings from known ones} We will consider two geometric operations which allow us to produce new bridge surfaces from existing ones. These are generalizations of stabilizations for Heegaard splittings. Following \cite{HS1}, the bridge surface $\Sigma$ will be called {\em stabilized} if there is a pair of 
compressing disks on opposite sides of $\Sigma$ that intersect in a single 
point. The bridge surface is called 
{\em perturbed} if there is a pair of 
bridge disks $D_i$ on opposite sides of $\Sigma$ 
such that $ \emptyset \neq (\bdd D_1 \cap \bdd D_2) \subset (\Sigma \cap T)$
and $|\bdd D_1 \cap \bdd D_2|=1$. These operations are discussed in detail in \cite{STo3}.

\subsection{Sweep-outs} Suppose $(M,T)=(H^+, \tau^+)\cup_{\Sss} (H^-, \tau^-)$. From the definition of a spine one can construct a map $f:M \to [-1,1]$ such that $f^{-1}(1)$ is isotopic to a spine of $(H^+,\tau^+)$, $f^{-1}(-1)$ is isotopic to a spine of $(H^-, \tau^-)$ and $f^{-1}(t)$ is a surface isotopic to the punctured surface $\Sss$ for every $t \in (-1,1)$. This function is called a \textit{sweep-out} representing $(\Sss,(H^-, \tau^-), (H^+, \tau^+))$.

We give a brief overview of how sweep-outs can be applied to study bridge surfaces for tangles in a 3-manifold. Further details can be found in \cite{T2}. Consider a tangle properly embedded in a 3-manifold with two bridge splittings. Let $f$ be a sweep-out representing the bridge splitting $(\Sss,(H^-, \tau^-), (H^+, \tau^+))$ and let $g$ be another sweep-out representing a second bridge splitting for $(M,T)$ which we denote $(\Sss',(H'^-, \tau'^-), (H'^+, \tau'^+))$.  

Consider the two parameter sweep-out $f \times g$ mapping $(M,T)$ into the square $[-1,1] \times [-1,1]$.  Each point $(s,t)$ in the square represents a pair of surfaces $\Sss_t = f^{-1}(t)-\eta(T)$ isotopic to the punctured surface $\Sss$ and $\Sss'_s = g^{-1}(s)-\eta(T)$ isotopic to $\Sss'$.  The {\em graphic} is the subset $\Ggg$ of the square consisting of all points $(s,t)$ where either $\Sigma_t$ is tangent to $\Sigma'_s$ or $\Sigma_t \cap \Sigma'_s$ contains a point of $T$.  We say that $f \times g$ is {\em generic} if it is stable on the complement of the spines and each arc $\{t\} \times [-1,1]$ and $[-1,1] \times \{s\}$ contains at most one vertex of the graphic. If $f\times g$ is generic then at each (valence four) vertex of $\Ggg$ there are two points of tangency, two points of $T$ in the intersection, or one of each.  By general position of the spines $f^{-1}(\pm 1)$ with the surface $\Sss'$, the graphic $\Ggg$ is incident to $\bdd I \times I$ in only a finite number of points corresponding to tangencies between $f^{-1}(\pm 1)$ and $\Sss'$. 

\subsection{Splitting and spanning sweep-outs} Suppose $f$ and $g$ are sweep-outs for $(M,T)$ and $f \times g$ is generic. Generalizing \cite{JJ}, for some fixed values of $s$ and $t$ we will say that $\Sss_t$ is \textit{mostly above} $\Sss'_s$ if each component of $\Sss_t \cap H'^-_s$ (if there are any) is contained in a disk or a once-punctured disk in $\Sss_t$. Similarly we will say that $\Sss_t$ is \textit{mostly below} $\Sss'_s$ if each component of $\Sss_t \cap H'^+_s$ is contained in a disk or once-punctured disk in $\Sss_t$.  We will say that $g$ \textit{spans} $f$ if there are values $t_+$, $t_-$ and $s$ for which $\Sss_{t_+}$ is mostly above $\Sss'_s$ and $\Sss_{t_-}$ is mostly below $\Sss'_s$. We will say that $g$ spans $f$ \textit{positively} if $t_- <t_+$ and \textit{negatively} otherwise. These conditions are shown at the top of Figure \ref{fig:splitspan}.
Note that $g$ may span $f$ both positively and negatively.

We will say that $g$ \textit{splits} $f$ if there is a value of $s$ such that the horizontal line $[-1,1] \times \{s\} \subset [-1,1]\times[-1,1]$ does not intersect any vertices of $\Gamma$ and for every $t$ the surface $\Sigma_t$ is neither mostly above nor mostly below $\Sigma'_s$.  This is shown at the bottom left of Figure \ref{fig:splitspan}. Note that this condition is equivalent to the condition that there exists an $s$ such that for every $t$, $\Sigma'_s \cap \Sigma_t$ contains at least one curve that is essential in $\Sss_t$.

\begin{figure}
\begin{center}
\includegraphics[scale=.5]{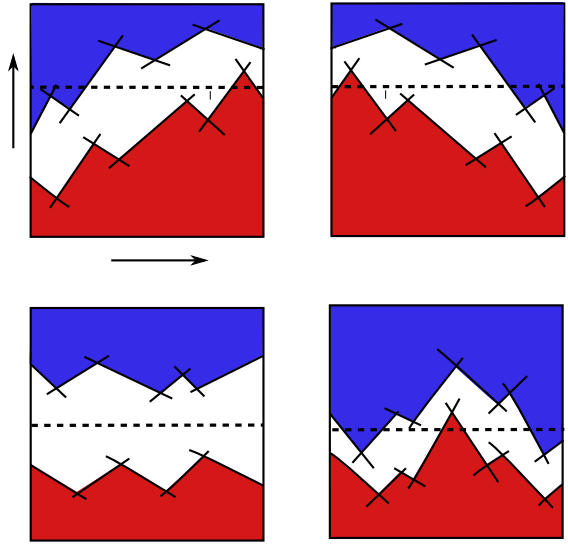}
\caption{The blue regions are those where $\Sigma_t$ is mostly below $\Sigma'_s$ and the red regions are where $\Sigma_t$ is mostly above $\Sigma'_s$. The figures show $g$ spanning $f$ positively, negatively, both positively and negatively and $g$ splitting $f$.}
\label{fig:splitspan}
\end{center}
\end{figure}

\section{Spanning sweep-outs and bounds on Euler characteristic} \label{sec:span}

As in the last section, we will let $f$ and $g$ be sweep-outs for the pair $(M,T)$ associated to the two bridge splittings $(\Sss,(H^-, \tau^-), (H^+, \tau^+))$ and $(\Sss',(H'^-, \tau'^-), (H'^+, \tau'^+))$ respectively. Define $\Sigma_t = f^{-1}(t)-\eta(T)$, and $\Sigma'_s = g^{-1}(s)-\eta(T)$ for $s,t \in(-1,1)$.  We will also name the compression bodies $H'^-_s = g^{-1}(-1,s]$ containing the trivial arcs $\tau'^-_s$ and $H'^+_s = g^{-1}[s,1)$ containing the trivial arcs $\tau'^+_s$.  

\begin{thm}\label{thm:span}
Let $f$ and $g$ be sweep-outs associated to bridge surfaces $\Sigma$ and $\Sigma'$ for a prime, unsplit tangle $T$ in an irreducible orientable 3-manifold $M$. Suppose that $f \times g$ is generic and suppose that there are values $s$ and $t_+>t_0$ such that $\Sss_{t_+}$ is mostly above $\Sss'_s$ and $\Sss_{t_0}$ is mostly below $\Sss'_s$. Then there is a sequence of compressions and cut compressions of $\Sigma'_s$ after which there is a component of the compressed surface which is parallel to $\Sigma$. If there are values $s$ and $t_+>t_0>t_-$ such that $\Sss_{t_+}$ and $\Sss_{t_-}$ are mostly above $\Sss'_s$ and $\Sss_{t_0}$ is mostly below $\Sss'_s$ then there is a sequence of compressions and cut compressions of $\Sigma'$ after which there are two components of the compressed surface which are parallel to $\Sigma$.
\end{thm}

Because the Euler characteristic is non-decreasing under c-compression, this theorem implies the following corollary: 

\begin{cor}
\label{cor:span}
Let $f$ and $g$ be sweep-outs associated to bridge surfaces $\Sigma$ and $\Sigma'$ for a prime, unsplit tangle $T$ in an irreducible orientable 3-manifold $M$. If $f \times g$ is generic and $g$ spans $f$ positively (or negatively) then $\chi(\Sigma')\leq \chi(\Sigma)$. If $g$ spans $f$ both positively and negatively then $\chi(\Sigma')\leq 2\chi(\Sigma)$.

\end{cor}

{\em Proof of Theorem \ref{thm:span}}
We will only prove the second statement as the proof of the first statement is similar but simpler. Suppose there are values $s$ and $t_+>t_0>t_-$ such that $\Sss_{t_+}$ and $\Sss_{t_-}$ are mostly above $\Sss'_s$ and $\Sss_{t_0}$ is mostly below $\Sss'_s$.

By definition, each loop in the intersection $\Sss'_s \cap (\Sssp \cup \Sss_{t_0} \cup \Sss_{t_-})$ bounds a disk or a once-punctured disk in $\Sssp \cup \Sss_{t_0} \cup \Sss_{t_-}$. To facilitate this discussion color ${H'_s}^-$ blue and ${H'_s}^+$ red. This induces a coloring on $\Sss_{t_+}$, $\Sss_{t_-}$ and $\Sss_{t_0}$. As $\Sss_{t_+}$ and $\Sss_{t_-}$ are mostly above $\Sss'_s$ every component of $(\Sss_{t_+} \cup \Sss_{t_-})-\Sss'_s$ that is not contained in a possibly once punctured sub-disk of $\Sss_{t_+}$ or $\Sss_{t_-}$ is red. Every component of $\Sss_{t_0}-\Sss'_s$ that is not contained in a possibly once punctured sub-disk of $\Sss_{t_0}$ is blue. 

Let $\ell$ be an innermost in $\Sssp$ curve of $\Sss'_s \cap \Sssp$. By hypothesis $\ell$ is necessarily inessential in $\Sssp$. If $\ell$ is also inessential in $\Sss'_s$ then it can be removed by an isotopy of $\Sss'_s$ as $M$ is irreducible and $T$ is prime. If $\ell$ is essential in $\Sigma'_s$ then the possibly punctured disk it bounds in $\Sssp$ is a c-disk for $\Sigma'_s$. Replace $\Sigma'_s$ with the surface $F_0$ that results from c-compressing $\Sigma'_s$ along this c-disk. Note that neither of these two moves affects the coloring of any region of $\Sssp-\Sss'_s$ that is not contained in a possibly punctured subdisk of $\Sssp$. We can then repeat this construction with an innermost loop of $F_0 \cap \Sssp$, producing a surface $F_1$ and so on until we find a surface $F_k$ disjoint from $\Sssp$. At the end of this sequence of isotopies and c-compressions, $\Sssp$ will be entirely red. 

Repeat the above process with $\Sss_{t_-}$ and $F_k$ playing the roles of $\Sssp$ and $\Sss'_s$ respectively to obtain a surface $F_\ell$ disjoint from both $\Sssp$ and $\Sss_{t_-}$ and leaving $\Sss_{t_-}$ entirely red. Finally repeat the process beginning with $\Sss_{t_0}$ and $F_\ell$ to obtain a surface $F_m$ disjoint from all of  $\Sssp \cup \Sss_{t_0} \cup \Sss_{t_-}$ and leaving $\Sss_{t_0}$ entirely blue. 

Maximally c-compress the surface $F_n=F_m \cap f^{-1}(t_-, t_+)$ in the complement of $\Sssp \cup \Sssz \cup \Sssm$ to get a surface $\tilde F$. Each component of $\tilde F$ is contained in a 3-manifold homeomorphic to $\Sss \times I$ and is c-incompressible in this manifold. By \cite[Corollary 3.7]{T1} each component of $\tilde F$ is either a sphere disjoint from $T$, a sphere bounding a ball containing a trivial subarc of $T$ or a component parallel to $\Sssz$. Note that $\tilde F$ was obtained from $\Sss'_s$ by c-compressions and therefore it cannot have sphere components disjoint from $T$ as $\Sss'_s$ does not have any such component, i.e., all component of $\tilde F$ have non-positive Euler characteristic. In addition $\tilde F$ separates $\Sssz$ from $\Sssp$ and  $\Sssz$ from $\Sssm$ as $\Sssz$ is entirely blue and $\Sssp$ and $\Sssm$ are red. Therefore $\tilde F$ must have at least two components parallel to $\Sssz$, one lying in the product region between $\Sssm$ and $\Sssz$ and one lying in the product region between $\Sssz$ and $\Sssp$. 
\qed

\section{Splitting sweep-outs and bounds on distance}\label{sec:split}

We briefly review the definition of distance of a bridge surface. For more details see \cite{T2}.

\begin{defin} \label{def:distance} Suppose $M$ is a compact, orientable, irreducible $3$--manifold containing a properly embedded tangle $T$ and suppose $(\Sss,(H^-, \tau^-), (H^+, \tau^+))$ is a bridge splitting for $(M,T)$.
The curve complex $\mathcal{C}(\Sigma,T)$ is
a graph with vertices corresponding to isotopy classes of
essential simple closed curves in $\Sigma-\eta(T)$. Two vertices are adjacent
in $\mathcal{C}(\Sigma,T)$ if their corresponding classes of curves
have disjoint representatives.

Let $\mathcal{V}^+$ (resp $\mathcal{V}^-$) be the set of all essential
 simple closed curves in $\Sigma-\eta(T)$ that bound disks in
 $H^+ -\eta(T)$ (resp $H^- -\eta(T)$). Then the {\em distance} of the bridge splitting, $d(\Sigma,T)$, is the minimum distance between a vertex in $\mathcal{V}^+$ and a vertex in $\mathcal{V}^-$ measured in
 $\mathcal{C}(\Sigma,T)$ with the path metric.
 \end{defin}

\begin{thm}\label{thm:split}
Let $f$ and $g$ be sweep-outs associated to bridge surfaces $\Sigma$ and $\Sigma'$ for a prime tangle $T$ in an irreducible 3-manifold $M$ and suppose that $\chi(\Sigma)\leq -1$. If $f \times g$ is generic and $g$ splits $f$ then $d(\Sss,T) \leq 2-\chi(\Sss')$.
\end{thm}

\begin{proof}
Let $s$ be such that for every $t\in(-1,1)$ the intersection $\Sigma'_s \cap \Sss_t$ contains a curve that is essential in $\Sss_t$ and $[-1,1]\times \{s\}$ is disjoint from all vertices of $\Gamma$ so in particular $g|_{\Sss_t}$ is Morse. Let $H_t^-=f^{-1}[-1,t)$ and  $H_t^+=f^{-1}(t,1]$ be the two components of $M-\Sigma_t$. For values of $t$ close to $-1$, all curves of $\Sss'_s \cap \Sss_t$ bound disks in $H_t^-$ because $\Sss'_s$ is transverse to the spine $f^{-1}(-1)$. Similarly for values of $t$ close to $1$, all curves of $\Sss'_s \cap \Sss_t$ bound disks in $H_t^+$. 

If for some value $t$ there is a curve of $\Sss' \cap \Sss_t$ that is essential in $\Sss_t$ and bounds a disk in $H_t^+$ and simultaneously there is a curve of $\Sss' \cap \Sss_t$ that is essential in $\Sss_t$ and bounds a disk in $H_t^-$, then $d(\Sss,T) \leq 1$. 

We will say that two values $t_-$, $t_+$ are \textit{adjacent} if there is a single critical value for $f$ between them.  In this case, the projections of the curves $\Sss' \cap \Sss_{t_-}$ to $\Sss_{t_+}$ can be isotoped disjoint from the curves $\Sss' \cap \Sss_{t_+}$.  Thus, if for some adjacent values $t_-$, $t_+$, there is a curve of $\Sss' \cap \Sss_{t_-}$ that is essential in $\Sss_{t_-}$ and bounds a disk in $H_{t_-}^-$ and a curve of $\Sss' \cap \Sss_{t_+}$ that is essential in $\Sss_{t_+}$ and bounds a disk in $H_{t_+}^+$, we can again conclude that $d(\Sss,T) \leq 1$. 

The above discussion shows that either $d(\Sss,T) \leq 1\leq 2-\chi(\Sigma')$ or there is an interval $[\alpha,\beta]$, where $\alpha \neq \beta$ are critical values for $f|_{\Sss'_s}$ such that for every $t \in (\alpha,\beta)$, no curve of $\Sss'_s \cap \Sss_t$ is both essential in $\Sss_t$ and bounds a disk in  $\Sss'_s$. Moreover for a very small $\epsilon$, $\Sss'_s \cap \Sigma_{\alpha-\epsilon}$ contains  curve that is essential in $\Sigma_{\alpha-\epsilon}$ and bounds a disk in $H^-_{\alpha-\epsilon}$ and $\Sss'_s \cap \Sigma_{\beta+\epsilon}$ contains a curve that is essential in $\Sigma_{\beta+\epsilon}$ and bounds a disk in $H^+_{{\beta+\epsilon}}$. 

Let $\alpha'$ be just above $\alpha$ and $\beta'$ be just below $\beta$. Suppose some component of $\Sss'_s \cap \Sss_{\alpha'}$ bounds a disk in $\Sss'_s$. Then this component must also bound a disk in $\Sss_{\alpha'}$ and therefore $\Sss'_s$ can be isotoped to remove this component. After some number of isotopies we obtains a surface $\Sss''$ so that no curve of $\Sss'' \cap \Sss_{\alpha'}$ or $\Sss'' \cap \Sss_{\beta'}$ bounds a disk in $\Sss''$.
Define $S = \Sss'' \cap f^{-1}[\alpha',\beta']$.  Because the boundary curves of $S$ do not bound disks in $\Sss''$, it follows that $\chi(S) \geq \chi(\Sss'_s)$. Let $\pi$ be the projection map from $f^{-1}[\alpha', \beta']$ to $\Sss_0$. By \cite[Lemma 22]{JJ}, isotopy classes of loops in $S$ project to isotopy classes in $\Sss_0$. Although we are now dealing with punctured surfaces the proof of this result is the same so we will not repeat it here.

As in \cite{JJ} we let $L$ be the set of isotopy classes of loops of $f|_S$ and let  $\pi^*$ be the natural map from $L$ to $\mathcal{C}(\Sigma_0,T)$, together with $\{0\}$ where each curve in $L$ maps to the vertex that corresponds to its projection in $\Sss_0$ unless it is inessential in $\Sss_0$ in which case it is mapped to $\{0\}$. 

Note that $L$ determines a decomposition of $S$ into pairs of pants and punctured annuli. Lemma 23 in \cite{JJ} shows that if $\ell$ and $\ell'$ are cuffs of the same pair of pants, then their images under $\pi^*$ are adjacent vertices in $\mathcal{C}(\Sss_0,T)$. The same is true if $\ell$ and $\ell'$ are the two boundary components of a punctured annulus. For if that is the case, then $f|_{\Sigma'}$ passes through a puncture so it contains a level component which is an arc with both of its endpoints lying in a boundary component of $\Sigma'$. The projection of this component to $\Sigma_0$ is also an arc with both endpoints on some boundary component. The boundary curves of a regular neighborhood of the arc together with the boundary component are isotopic to the projections of $\ell$ and $\ell'$ and thus $\ell$ and $\ell'$ are disjoint.

Let $L'=\pi_*(L)\cap \mathcal{C}(\Sss_0,T)$. By \cite[Lemma 24]{JJ} this set is connected and has diameter equal to at most the number of components of $S-L$. Each component of $S-L$ is a punctured annulus or a pair of pants and therefore contributes $-1$ to $\chi(S)$. It follows that $diam(L')\leq   -\chi(S)$.

Recall that for a very small $\epsilon$, $\Sss'\cap \Sigma_{\alpha-\epsilon}$ contains a curve that bounds a compressing disk for $H_{\alpha-\epsilon}^-$ and $\Sss'\cap \Sigma_{\beta+\epsilon}$
contains a curve that bounds a compressing disk for $H_{\beta+\epsilon}^+$. As the intervals $(\alpha-\epsilon,\alpha')$ and $(\beta', \beta+\epsilon)$ contain exactly one critical point each, every curve in the set $\pi( \Sss'\cap \Sigma_{\alpha-\epsilon})$ is distance at most one from every curve in the set $\pi( \Sss'\cap \Sigma_{\alpha'})$ and similarly every curve in the set $\pi( \Sss'\cap \Sigma_{\beta+\epsilon})$ is distance at most one from every curve in the set $\pi( \Sss'\cap \Sigma_{\beta'})$. Adding these distances we obtain the inequality $d(\Sigma)\leq diam(L')+2\leq 2-\chi(S)\leq 2-\chi(\Sigma')$ as desired. 
\end{proof}

In this and in the previous section we saw that if $f$ and $g$ are two sweep-outs associated to bridge surfaces $\Sigma$ and $\Sigma'$ for the pair $(M,T)$ and $g$ spans $f$, then we can relate $\chi(\Sigma)$ and $\chi(\Sigma')$ and if $g$ splits $f$ then we can relate $d(\Sss,T)$ and $\chi(\Sss')$. It is clear that if $g$ and $f$ are sweep-outs such that $f \times g$ is generic, then either $g$ spans $f$, $g$ splits $f$ or there is are values of $s$ and $t$ such that for a small $\epsilon$, $\Sss_{t}$ is mostly above $\Sss'_{s-\epsilon}$ and $\Sss_{t}$ is mostly below $\Sss'_{s+\epsilon}$. We now consider a slight generalization of this last case.

\begin{lemma}\label{lem:other}
Suppose $f$ and $g$ are sweep-outs for a tangle in a manifold such that $f \times g$ is generic except possibly for a single vertex of order 6 or two vertices of order 4 with the same $s$ coordinate. Suppose the graphic of $f \times g$ has a vertex at coordinates $(s,t)$ such that for a small $\epsilon$, $\Sss_{t}$ is mostly above $\Sss'_{s-\epsilon}$ and $\Sss_{t}$ is mostly below $\Sss'_{s+\epsilon}$. If this vertex has valence 4, then $\chi(\Sss)\geq -2$. If the vertex has valence 6, then $\chi(\Sss)\geq -3$.
\end{lemma}

\begin{proof}
By the definition of $f \times g$ it follows that, $g|_{\Sss_{t'}}$ is Morse where $t'=t+\epsilon'$ for a small $\epsilon'$. Furthermore  there are two critical values for $g|_{\Sss_{t'}}$, $a < b$ with at most one other critical value between them (if the valence of $(s,t)$ is 6) such that if $a'$ is a regular value directly below $a$ and $b'$ is a regular value directly above $b$, then $\Sss_{t'}$ is mostly above $\Sss'_{a'}$ and mostly below $\Sss'_{b'}$. 

Consider first $\Sss_{t'} \cap \Sss'_{a'}$. By definition each component of $\Sss_{t'} \cap H'^-_{a'}$ is contained in a possibly punctured disk subset of $\Sss_{t'}$. Let $\Lambda$ be the set of all curves of $\Sss_{t'} \cap \Sss'_{a'}$ that are not contained in the interior of a disk or punctured disk component of $\Sss_{t'} - \Sss'_{a'}$, see Figure \ref{fig:curves}. Then $\Sss_{t'}-\Lambda$ is a collection of components all but one of which are possibly punctured disks. Note that the Euler characteristics of each of these possibly punctured disk components is at least 0. 

Passing through each critical point between $a'$ and $b'$ is equivalent to adding a band between two components of $\Sss_{t'}-\Sss'_{a'}$ or banding a component to itself. In either case the sum of the Euler characteristics of all components is decreased by one. As these bands correspond to a sweep-out they all lie on the same side of $\Sss'$. As $\Sss_{t'}$ is mostly below $\Sss'_{b'}$, it follows that after attaching at most three bands to a collection of at most once punctured disks, the result is a surface isotopic to $\Sss_{t'}$ with possibly some disks and once punctured disks missing, i.e. $\Sss_{t'}\cap \Sss'_{b'}$ is also as in Figure \ref{fig:curves} but now the subsurface which is not contained in a punctured disk is below $\Sss'_{b'}$. As at most three bands were added, it follows that $\chi(\Sss_{t'})\geq -3$. If the vertex $(s,t)$ has valence four, then only two bands need to be added so $\chi(\Sss_{t'})\geq -2$.

\begin{figure}
\begin{center}
\includegraphics[scale=.35]{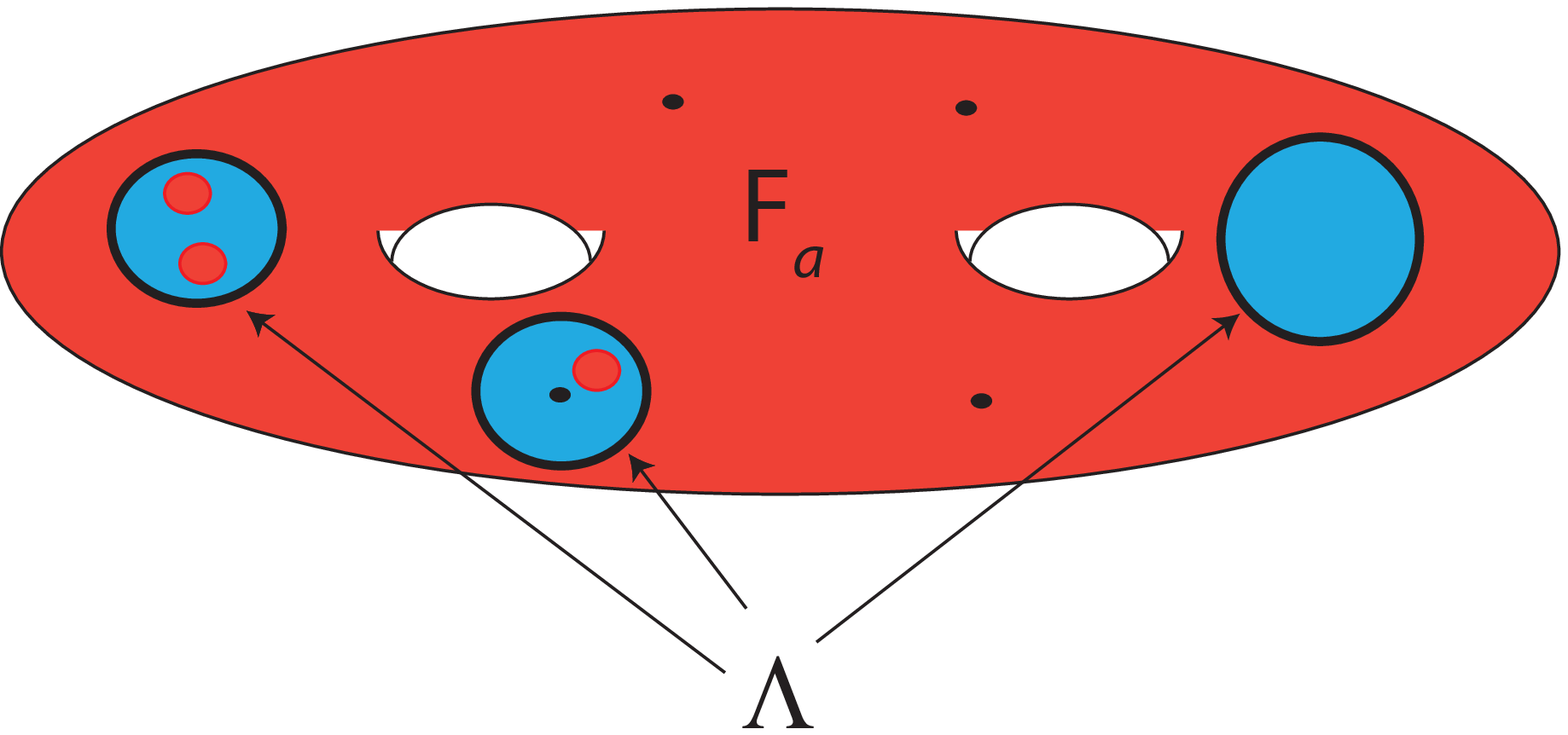}
\caption{}
\label{fig:curves}
\end{center}
\end{figure}

\end{proof}

Using the results in this and the previous section we can obtain the following generalization of the main result in \cite{T2}.

 \begin{thm}\label{thm:boundbridge}
Suppose $N$ is a manifold containing a tangle $K$ and let $M$ be submanifold such that $T=K \cap M$ is a properly embedded tangle. Let $\Sigma$ be a bridge surface of $(M,T)$ and let $\Sigma'$ be a bridge surface of $(N,K)$. Then one of the following holds:

\begin{itemize}
\item There is an isotopy to $\Sss'$ followed by some number of compressions and cut-compressions of $\Sigma'\cap M$ in $M$ giving a compressed surface $\Sigma''$, such that at least one component of $\Sigma'' \cap M$ is parallel to $\Sigma$,
\item $d(\Sigma,T) \leq 2-\chi(\Sigma')$,
\item $\chi(\Sigma)\geq -2 $.
\end{itemize}
\end{thm}

\begin{proof}
Because $(M,T) \subset (N,K)$ for values of $s$ close to $-1$, $\Sigma_t$ is mostly above $\Sigma'_s$ and for values of $s$ close to $1$, $\Sigma_t$ is mostly below $\Sigma'_s$. Therefore there are three possibilities. Either $g$ spans $f$, $g$ splits $f$ or the graphic of $f \times g$ has a vertex of valence 4 at coordinates $(s,t)$ such that for a small $\epsilon$, $\Sss_{t}$ is mostly above $\Sss'_{s-\epsilon}$ and $\Sss_{t}$ is mostly below $\Sss'_{s+\epsilon}$. Thus there are three cases to consider.

\textbf{Case 1:} If $g$ spans $f$ then there are values $s$ and $t_+$ and $t_-$ such that $\Sss_{t_+}$ is mostly above $\Sss'_s$ and $\Sss_{t_-}$ is mostly below $\Sss'_s$. By the arguments in Theorem \ref{thm:span} it follows that after some number of compressions and cut-compressions of $\Sigma_s'\cap M$ in $M$ we obtain an incompressible surface $\Sigma''$ that separates $\Sss_{t_+}$ and $\Sss_{t_-}$ and therefore it is parallel to $\Sigma$ as desired. 

\textbf{Case 2:} If $g$ splits $f$, then the arguments in the proof of Theorem \ref{thm:split} show that 
 $d(\Sigma,T)\leq 2-\chi(\Sigma')$ as desired. 

\textbf{Case 3:} Finally suppose that the graphic of $f \times g$ has a vertex at coordinates $(s,t)$ such that for a small $\epsilon$, $\Sss_{t}$ is mostly above $\Sss'_{s-\epsilon}$ and $\Sss_{t}$ is mostly below $\Sss'_{s+\epsilon}$. Then by Lemma \ref{lem:other} it follows that $\chi(\Sigma)\geq -2$.

\end{proof}

\section{Flipping bridge surfaces}\label{sec:flipping}

In this section we want to restrict our attention to oriented isotopies, i.e., if $\Sss$ and $\Sss'$ are bridge splittings for $(M,T)$ splitting the manifold into compression bodies $H^+, H^-$ and $H'^+, H'^-$ respectively, the bridge splittings $(\Sss, (H^+, \tau^+), (H^-, \tau^-))$ and $(\Sss', (H'^+, \tau'^+), (H'^-, \tau'^-))$ will be called {\em orientation isotopic} if there is an isotopy mapping $\Sss$ to $\Sss'$, $(H^+, \tau^+)$ to $(H'^+, \tau'^+)$ and $(H^-, \tau^-)$ to $(H'^-, \tau'^-)$. Following \cite{JJ} we will say that a bridge surface $\Sss$ is {\em flippable} if $(\Sss, (H^+, \tau^+), (H^-, \tau^-))$ is orientation isotopic to $(\Sss, (H^-, \tau^-), (H^+, \tau^+))$. 

Suppose $(\Sss', (H'^+, \tau'^+), (H'^-, \tau'^-))$ is a bridge splitting for $(M,T)$ isotopic to stabilizations and perturbations of both bridge splittings $(\Sss, (H^+, \tau^+), (H^-, \tau^-))$ and $(\Sss, (H^-, \tau^-), (H^+, \tau^+))$. The minimal value of $2-\chi(\Sss')$ is called the {\em flip Euler characteristic of $\Sss$} and it is analogous to the flip genus of a Heegaard splitting defined in \cite{JJ}. 

We will take advantage of several results previously proven for sweep-outs of Heegaard splittings. The proofs carry over with only minor modifications.

\begin{lemma} \label{lem:span}  (See Lemma 9 and Lemma 11 in \cite{JJ})
If $(\Sigma,(H^+,\tau^+), (H^-,\tau^-))$ is a bridge decomposition for some knot $K \subset M$ then $(\Sigma,(H^+,\tau^+), (H^-,\tau^-))$ spans itself positively. If $(\Sigma',(H'^+,\tau'^+),(H'^-,\tau'^-))$ is a perturbation or stabilization of $(\Sigma,(H^+,\tau^+),(H^-,\tau^-))$, then it spans the bridge splittings $(\Sigma, (H^+,\tau^+),(H^-,\tau^-))$ positively and spans $(\Sigma, (H^-,\tau^-), (H^+,\tau^+))$ negatively.
\end{lemma}
\begin{proof}

Let $f$ be a sweep-out for the decomposition $(\Sigma,(H^+,\tau^+), (H^-,\tau^-))$. Then there is a second sweep-out $g$ for $(\Sigma, (H^+,\tau^+),(H^-,\tau^-))$ such that $g^{-1}(0)=\Sigma'_0$ is disjoint from and separates the spines $f^{-1}(-1)$, $f^{-1}(1)$.  Thus for $t_-$ near $-1$ and $t_+$ near $1$, $\Sigma_{t_-}$ will be mostly below $\Sigma'_0$ and $\Sigma_{t_+}$ will be mostly above.  This implies that $(\Sigma,(H^+,\tau^+), (H^-,\tau^-))$ spans itself positively. 

Similarly, for any perturbation or stabilization of the bridge decomposition, we can perturb or stabilize $\Sigma'_0$ while keeping it disjoint from $f^{-1}(-1)$, $f^{-1}(1)$, then extend this surface to a sweep-out for the perturbed or stabilized bridge decomposition.  Thus if $(\Sigma',(H'^+,\tau'^+),(H'^-,\tau'^-))$ is a perturbation or stabilization of $(\Sigma,(H^+,\tau^+),(H^-,\tau^-))$, then it spans the bridge splitting $(\Sigma, (H^+,\tau^+),(H^-,\tau^-))$ positively and spans $(\Sigma, (H^-,\tau^-), (H^+,\tau^+))$ negatively.

\end{proof}
Let $\Sss'$ be a common stabilization or perturbation of the two bridge splittings $(\Sss,(H^-, \tau^-), (H^+, \tau^+))$ and $(\Sss,(H^+, \tau^+), (H^-, \tau^-))$. In particular there are sweep-outs $g$ and $g'$ representing $(\Sigma',(H'^-, \tau'^-), (H'^+, \tau'^+))$, and a sweep out $f$ representing  $(\Sss,(H^-, \tau^-), (H^+, \tau^+))$ such that $g$ spans $f$ positively and $g'$ spans $f$ negatively. As $g$ and $g'$ represent the same bridge decomposition and are therefore orientation isotopic it follows that there is a family of sweep-outs $\{g_r|r \in [0,1]\}$ such that $g_0=g$, $g_1=g'$.

\begin{figure}
\begin{center}
\includegraphics[scale=.5]{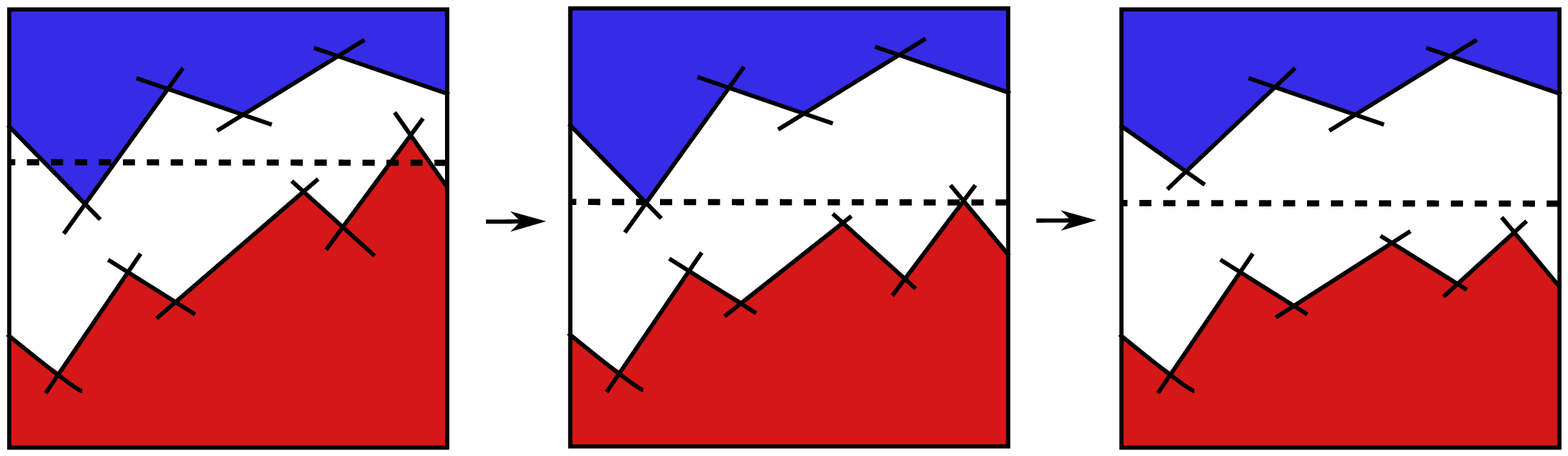}
\caption{}\label{fig:span}
\end{center}
\end{figure}

\begin{lemma}\cite[Lemma 26]{JJ}\label{lem:iso}
Let $g$ and $g'$ be sweep-outs such that $ f\times g$ and $f \times g'$ are generic and $g$ is isotopic to $g'$. There is a family of sweep-outs $\{g_r|r \in [0,1]\}$ such that $g_0=g$, $g_1=g'$ and for all but finitely many $r$, $f \times g_r$ is generic. At the finitely many non-generic points there are at most two valence two or four vertices at the same level or there is a single valence 6 vertex.
\end{lemma}  

We can now prove our first main result.

\begin{proof}[Proof of Theorem \ref{thm:flip}]

Consider the family of sweep-outs $\{g_r|r \in [0,1]\}$ described in Lemma \ref{lem:iso}. As $g_0$ spans $f$ positively and $g_1$ spans $f$ negatively there must be some $r$ such that either $g_r$ splits $f$, $g_r$ spans $f$ both positively and negatively, or the hypothesis of Lemma \ref{lem:other} are satisfied. This is illustrated in Figure \ref{fig:span}.

\medskip
Case 1: $g_r$ splits $f$. In this case by Theorem \ref{thm:split} it follows that $d(\Sss,T)\leq 2-\chi(\Sigma')$ so $\chi(\Sss')\leq 2-d(\Sss,T)$. 

\medskip
Case 2: $g_r$ spans $f$ both positively and negatively. In this case by Theorem \ref{thm:span} it follows that $\chi(\Sigma)\leq  2\chi(\Sigma')$.

\medskip 
Case 3: There are at most two valence two or four vertices at the same level or there is a valence 6 vertex. By an argument identical to the one in the proof of  \cite[Lemma 26]{JJ} it follows that either we are in one of cases 1 or 2 or there is a vertex of valence 4 or valence 6 corresponding to coordinates $(s,t)$ such that for a very small $\epsilon$ the surface $f^{-1}_{t+\epsilon}$ is mostly above $g^{-1}(s)$ and $f^{-1}_{t-\epsilon}$ is mostly below $g^{-1}(s)$. Lemma \ref{lem:other} shows that in this case $\chi(\Sss)\geq -3$ contradicting the hypothesis.
\end{proof}

\begin{figure}
\begin{center}
\includegraphics[scale=.65]{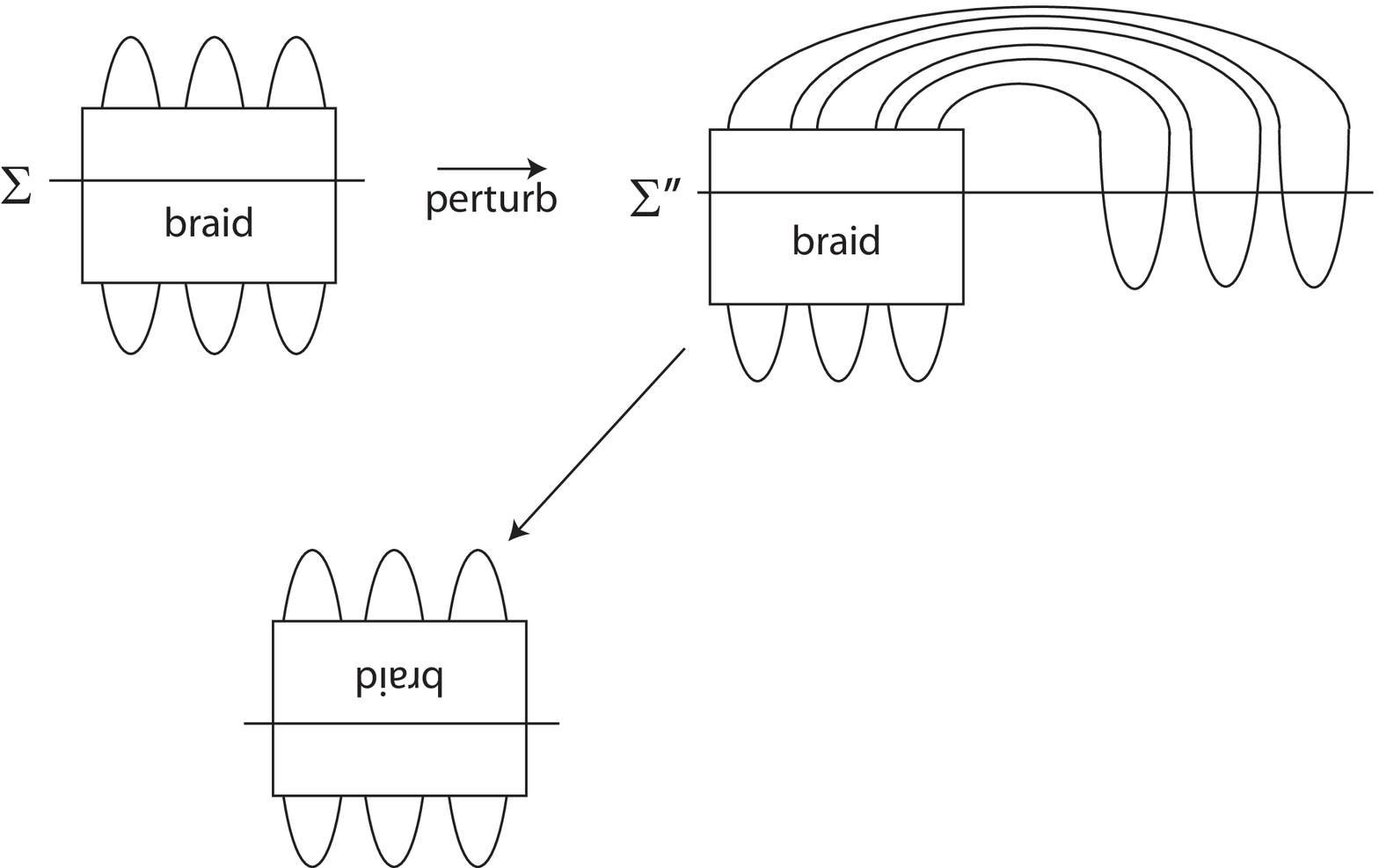}
\caption{}\label{fig:flip}
\end{center}
\end{figure}
\medskip

{\em Proof of Corollary \ref{cor:flip}}
Figure \ref{fig:flip} shows that if $T$ has $n$ bridges with respect to a bridge sphere $\Sss$, then there is a flippable bridge sphere $\Sigma'$ obtained from $\Sss$ by perturbations with respect to which $T$ has $2n$ bridges. The fact that there is no such bridge sphere with fewer punctures follows by Theorem \ref{thm:flip}.
\qed

\section{Bridge surfaces that require a large number of stabilization and perturbations to become equivalent}
\label{sec:example}

{\em Proof of Theorem \ref{thm:knot}}. The proof of this theorem consists of a construction for a pair $(M,K)$ where $K$ is a knot with two distinct bridge surfaces, $\Sss$ and $\Sss'$ with $\chi(\Sss)=2s$ and $\chi(\Sss')=2s-2$ so that for every common stabilization/perturbation $\Sss''$ of $\Sss$ and $\Sss'$, $\chi(\Sss')\leq 3s+2$. In particular this construction gives examples of knots in $S^3$ with distinct bridge spheres with bridge number $2n$ and $2n-1$, respectively, for which any common perturbation has at least $3n-1$ bridges.

Let $K$ be a knot in a manifold $M$ and let $(\Sigma,(H^+,\tau^+),(H^-,\tau^-))$ be a bridge splitting for $(M,K)$ so that, $\chi(\Sigma)\leq -4$ and  $d(\Sigma,K)\geq-3\chi(\Sigma)$. Suppose $f$ is a sweep-out for $M$ associated to $\Sigma$. Let $\Gamma_+$ be a spine of $(H^+,\tau^+)$ and let $\Gamma_-$ be a spine of $(H^-,\tau^-)$ so that $\Gamma_-=f^{-1}(-1)$ and $\Gamma_+=f^{-1}(1)$. Choose an edge of $\Gamma_-$ that has a valence 1 vertex; i.e., an edge that has one endpoint in $K$. Let $B$ be a ball that is a regular neighborhood of this edge and let $M^-$ be the closure of $M \setminus B$ containing the one strand tangle $K^-=K \cap M^-$. The sweep-out $f$ on $M$ can be modified to be a sweep-out of $M^-$ by perturbing $f$ to be constant in $B$. We will use $f$ to refer to either sweep-out when the manifold is clear from context.

Let $P$ be a manifold homeomorphic to $S^2 \times I$ containing two vertical arcs $\tau_1$ and $\tau_2$. Construct a new manifold $M\#M$ by gluing a copy of $M^-$ to each of the boundary sphere of $P$ so that the endpoints of each copy of $K^-$ are identified with one endpoint of $\tau_1$ and one endpoint of $\tau_2$ to obtain a new knot $K\#K$. Then $(M\#M,K\#K)$ is the connect sum of two copies of $(M,K)$.

The pair $(M\#M, K\#K)$ has two natural generalized Heegaard splittings $\mathcal{H}_1$ and $\mathcal{H}_2$ induced by the bridge splittings for $M$ and $P$, shown in Figures~\ref{fig:generalized} and~\ref{fig:composite}. In both cases we will take $\Sigma$ to be the bridge surface for each copy of $(M^-,K^-)$. However in the first generalized Heegaard splitting we will take the surface $S^2 \times \{1/2\}$ to be the bridge surface for $(P, \tau_1 \cup \tau_2)$ and for the second one we will take the bridge surface for $P$ to be the surface obtained by tubing together the two spheres which are boundaries of small collars of $S^2 \times \{0\}$ and $S^2 \times \{1\}$ respectively along a vertical tube, see Figure \ref{fig:generalized}. Let $\Sigma_1$ and $\Sigma_2$ be the two bridge surfaces for $(M\#M,K\#K)$ obtained by amalgamating $\mathcal{H}_1$ and $\mathcal{H}_2$ respectively. Note that $\chi(\Sigma_1)=2\chi(\Sigma)$ and $\chi(\Sigma_2)=2\chi(\Sigma)-2$.

\begin{figure}
\begin{center}
\includegraphics[scale=.5]{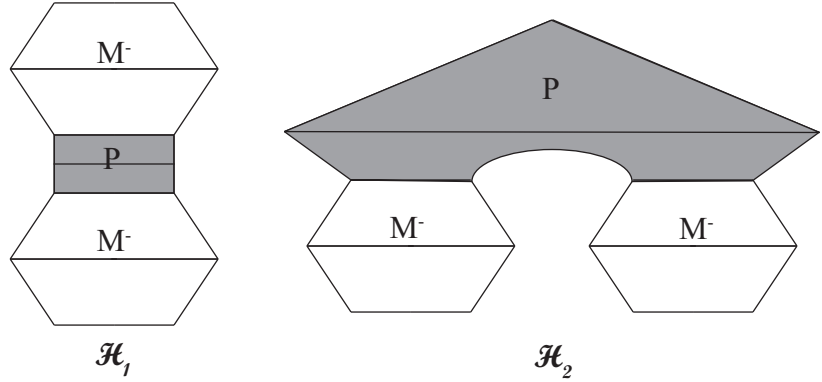}
\caption{Schematic depiction of $\mathcal{H}_1$ and $\mathcal{H}_2$.}\label{fig:generalized}
\end{center}
\end{figure}

\begin{figure}
\begin{center}
\includegraphics[scale=.5]{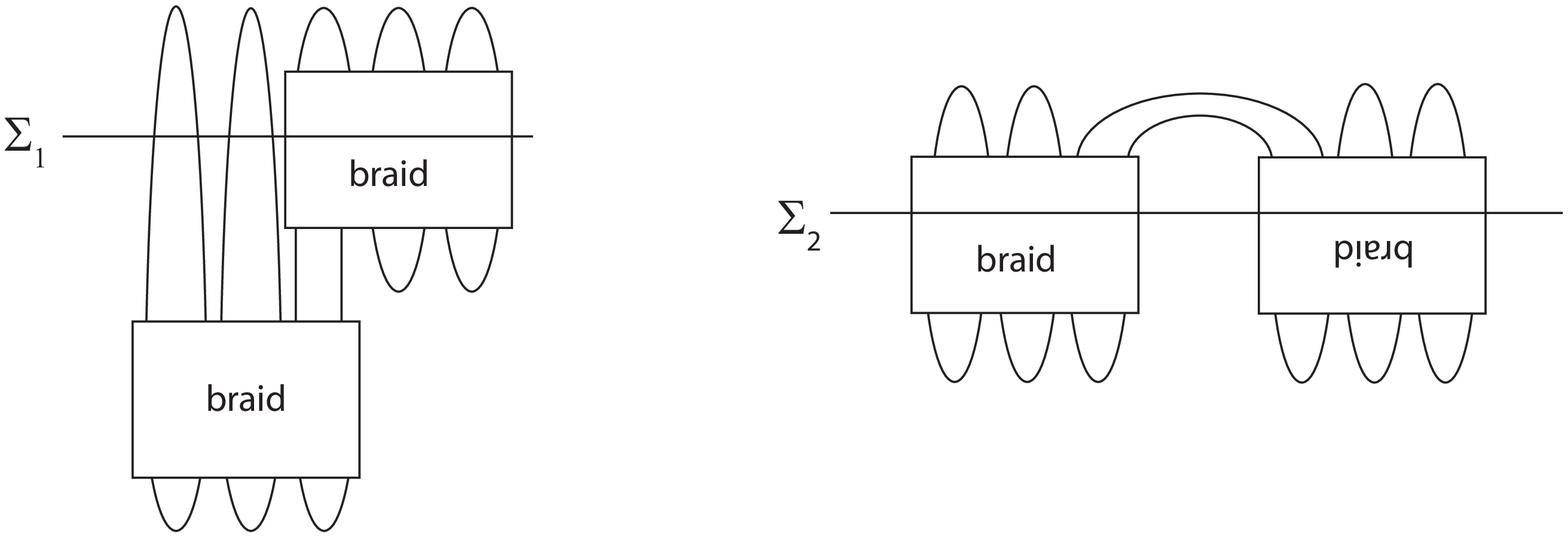}
\caption{The figure depicts $\Sigma_1$ and $\Sigma_2$ if $M=S^3$ and $\Sigma=S^2$.}\label{fig:composite}
\end{center}
\end{figure}

\begin{lemma}\label{lem:example}
If $\Sigma'$ is isotopic to a surface obtained via a sequence of stabilizations and perturbations of $\Sigma_1$ and also to a surface obtained via sequence of stabilizations and perturbations of $\Sigma_2$, then $\chi(\Sigma') \leq 3\chi(\Sigma)+2$.
\end{lemma}

\begin{proof}
Let $\mathcal {H}^-_i$ and $\mathcal {H}^+_i$ be the thin and thick surfaces for $\mathcal{H}_i$. Then $M\#M-\mathcal{H}^-_i$ has two components that are homeomorphic to $M^-$ and each of these components has a Heegaard surface $\Sigma$. Let $M^-_1$ and $M^-_2$ be these components and let $f_1$ and $f_2$ be their sweep-outs associated to $\Sigma$. 

The following remark is clear in the case when $M=S^3$ and $\Sigma=S^2$, as shown in Figure \ref{fig:composite}. In the general case the proof is very similar to the proof of \cite[Lemma 14]{JJ} so we leave the details to the reader. 
 
\begin{rmk}[See Lemma 14, \cite{JJ}] \label{rmk:amalgamation}
Let $\fp$ be the sweep-out associated to $\Sigma_1$ and $\fm$ be the sweep-out associated to $\Sigma_2$. Then $\fp$ spans $f_1$ and $f_2$ positively and $\fm$ spans $f_1$ positively and $f_2$ negatively. \end{rmk}

Let $g$ and $g'$ be sweep-outs for $(M\#M,K\#K)$ defined by perturbing and stabilizing sweep-outs $\fp$ and $\fm$ enough times so that $g$ and $g'$ represent isotopic bridge decompositions. By Lemma \ref{lem:span} it follows that $g$ spans $\fp$ positively. By Remark \ref{rmk:amalgamation} it follows that $\fp$ spans both $f_1$ and $f_2$ positively. Therefore we conclude that $g$ spans both $f_1$ and $f_2$ positively. Similarly, $g'$ spans $f_1$ positively and $f_2$ negatively.

As $g$ and $g'$ represent isotopic bridge decompositions, the sweep-out $g$ is isotopic to either $g'$ or $-g'$.  In other words, there is a family of sweep-outs $\{g_r|r \in [0,1]\}$ such that $g_0=g$, $g_1=\pm g'$. Consider the family of sweep-outs $\{g_r|r \in [0,1]\}$ described in Lemma \ref{lem:iso}. Because $g'$ spans $f_1$ positively and $f_2$ negatively, the sweep-out $g_1 = \pm g'$ spans one of $f_1$ or $f_2$ positively and the other negatively.  Without loss of generality, assume $g_1$ spans $f_1$ negatively.  As $g_0$ spans $f_1$ positively and $g_1$ spans $f_1$ negatively, Lemma~\ref{lem:iso} implies that there is an $r$ satisfying one of the following:

\medskip
Case 1: $g_r$ splits $f_1$ or $g_r$ splits $f_2$. The argument is the same so suppose $g_r$ splits $f_1$. In this case by Theorem \ref{thm:split} it follows that $d(\Sigma_1,K) \leq 2-\chi(\Sigma')$. As $d(\Sigma_1,K) \geq -3\chi(\Sigma)$ by construction, it follows that $\chi(\Sigma') \leq 2+4\chi(\Sigma)\leq 3\chi(\Sigma)+2$ as required. 

\medskip
Case 2: $g_r$ spans $f_1$ both positively and negatively and $g_r$ spans $f_2$, say positively.  

By the definition of spanning there exist $s$ and $t_+>t_0>t_-$ such that $(\Sigma_1)_{t_+}$ and $(\Sigma_1)_{t_-}$ are mostly above $\Sss'_s\cap M^-_1$ and $(\Sigma_1)_{t_0}$ is mostly below $\Sss'_s\cap M^-_1$ and there exist $u$ and $t'_0 < t'_+$ so that $(\Sigma_2)_{t'_0}$ is mostly below $\Sss'_u \cap M^-_2$ and $(\Sigma_2)_{t'_+}$ is mostly above $\Sss'_u\cap M^-_2$. If we can choose $s$ and $u$ to be equal, then by Theorem \ref{thm:span} it follows that $\chi(\Sss'_u)=\chi(\Sss'_u \cap M^-_1)+\chi(\Sss'_u\cap M^-_2)\leq 2\chi(\Sigma_1)+\chi(\Sigma_2)=3\chi(\Sigma)$ as desired. Suppose that no such value exists. Without loss of generality suppose that $s < u$ and choose $s$ and $u$ to be such that $u-s$ is minimal. By the choice of $s$ and Theorem \ref{thm:span} it follows that $\chi(\Sss'_s\cap M^-_1 )\leq 2\chi(\Sigma_1)$. 

Let $S$ be the decomposing sphere for $(M\#M, K\#K)$ (we may take $S=S^2\times \{1/2\}$ in $P$). The surface $\Sss'_u$ can be obtained from $\Sss'_s$ by a series of boundary compressions of $\Sigma'_u \cap M^-_1$ and $\Sigma'_u \cap M^-_2$ together with isotopies of the surface that are the identity in a neighborhood of $S$. Let $\Sigma'_0, \Sigma'_1,...,\Sigma'_n$ be the sequence of surfaces so that $\Sigma'_0$ is isotopic relative to $S$ to $\Sigma'_s$,  $\Sigma'_n$ is isotopic relative to $S$ to $\Sigma'_u$ and $\Sigma'_k$ is obtained from $\Sigma'_{k-1}$ by performing a boundary compression of $\Sigma'_{k-1} \cap M^-_1$ or a boundary compression of $\Sigma'_{k-1} \cap M^-_2$ along a disk $D_k$. Following \cite{J} we will call these boundary compressions $\alpha$-isotopies of $\Sigma'$ and we will call the boundary compressing disks, $\alpha$-disks. As the isotopy between $\Sigma'_s$ and $\Sigma'_u$ represents a sweep-out, all $\alpha$-disks are on the same side of $\Sss'$, in this case, the positive side as $s <u$.

Every $\alpha$-disk $D_k$ has a dual $\alpha$-disk $E_k$ contained in the negative side of $\Sigma'_k$ that can be used to perform an $\alpha$-isotopy on $\Sigma'_{k}$ to recover $\Sigma'_{k-1}$, as in Figure \ref{fig:dual}.  

\begin{figure}
\begin{center}
\includegraphics[scale=.5]{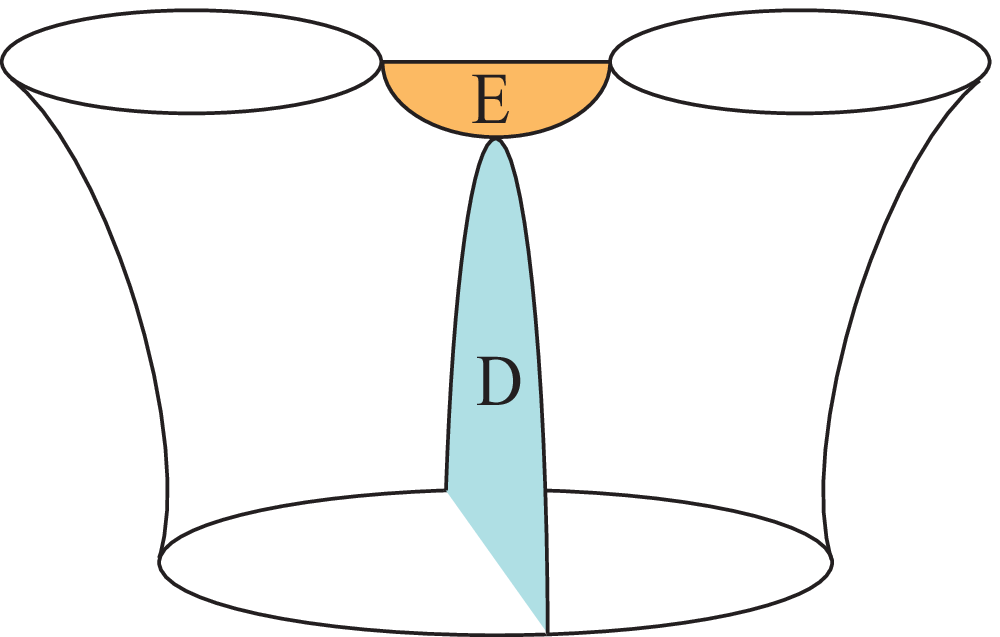}
\caption{}\label{fig:dual}
\end{center}
\end{figure}

\textbf{Claim:} There is a collection of disks $E_1,...,E_n$ such that for every $k$ and $j$ the following hold

\begin{itemize}
\item $\bdd E_k$ is the endpoint union of an arc in $S$ and an arc in $\Sss'_u$,
\item $E_k \cap E_j$ is either empty or it is equal to some disk $E_l$,
\item one component of the boundary of an $\epsilon$-neighborhood of $\Sss'_u \cup (E_n\cup...\cup E_{k+1})$ is isotopic to $\Sss'_u$ relative to $S$ and the other is isotopic to $\Sss'_k$ relative to $S$,
\item $E_k$ is an $\alpha$-disk for $\Sigma'_k$.

\end{itemize}
In particular if $\mathcal{E}=\cup_1^n E_i$ then one component of the boundary of an $\epsilon$-neighborhood of $\Sss'_u \cup \mathcal{E}$ is isotopic to $\Sss'_u$ relative to $S$ and the other is isotopic to $\Sss'_s$ relative to $S$.

{\em Proof of claim:} Let $\mathcal{E}_j=\cup_j^n E_j$. We will prove the claim by induction on $n-j$. The result is clear for $\mathcal{E}_n$ as this collection contains a single disk. Suppose the result holds for $\mathcal{E}_{k+1}$. Therefore $\Sigma'_k$ is isotopic relative $S$ to one of the two boundary components of a regular neighborhood of $\Sss'_u \cup \mathcal{E}_{k+1}$. 
Consider a disk $D_k$ that realizes the $\alpha$-isotopy between $\Sss'_{k-1}$ and $\Sss'_k$. Then the boundary of a regular neighborhood of $\Sss'_{k-1}\cup D_k$ has two components, one is $\Sss'_{k-1}$ and the other is $\Sss'_k$. If $D_k \cap \mathcal{E}_{k+1}=\emptyset$, let $E_k$ be the disk dual to $D_k$. If $D_k$ has a nonempty intersection with each of $E_{i_1},..., E_{i_l}$, let $\gamma_{i_r}$ be a small neighborhood of $D \cap E_{i-r}$ in $\bdd E_{i+r}$. Identify all arcs $\gamma_{i_r}$ for $r=1,...,l$ to a single arc $\gamma$ and let $E_k$ be the disk dual to $D_k$ so that $E_k \cap \Sigma'=\gamma$. 

\medskip

Consider the surface $\Sigma'_u \cap M^-_2$. By Theorem \ref{thm:span} there exists a collection ${\bf D}$ of compressing and cut-compressing disks for $\Sigma'_u \cap M^-_2$ after which the resulting union of surfaces $(\Sigma'_u)^{\bf D} \cap M^-_2$ contains a component parallel to $\Sigma_2$.  We can choose these disks so that a subcollection $\bdd\mathcal{D}_u$ is contained in $S$ and so that compressing $\Sigma'_u$ along the collection $\bdd\mathcal{D}_u$ produces a surface disjoint from $S$.  Note that $\chi((\Sigma'_u)^{\bf D}\cap M^-_2) \leq \chi(\Sigma)$

The surface $\Sigma'_u$ is isotopic to the union of $(\Sigma'_u)^{\bf D}$ together with some tubes between the components. Some of these tubes may run along the knot. Let $\mathcal{D}_u\subset {\bf D}$ be the collecting of possibly nested disks and punctured disks these tubes bound in $S$. Then $\mathcal{E}\cap S$ is a collection of disjoint arcs $\Lambda$ with endpoints in $\bdd\mathcal{D}_u$. Label the regions of $S-\Sigma'_u$ positive or negative depending on which side of $\Sigma'_u$ they lie in and recall that the collection of arcs $\Lambda$ lies in the negative regions. 

The surface $S$ intersects the knot $K$ in two points, which we will label $p_1$, $p_2$.  The curves in $\Sigma_u' \cap S$ can be classified into three categories depending on whether the corresponding c-disk in $\mathcal{D}_u$ is a disk, a punctured disk containing $p_1$, or a punctured disk containing $p_2$. Let $\Gamma_1$, $\Gamma_2$ and $\Gamma_3$ be these collection of curves and let $\Gamma=\Gamma_1\cup\Gamma_2 \cup \Gamma_3$. Note that the only arcs in $\mathcal{E}\cap S$ that have endpoints in both $\Gamma_2$ and $\Gamma_3$ are between the outermost curve $\gamma_2$ in $\Gamma_2$ and the outermost curve $\gamma_3$ in $\Gamma_3$. Also note that if an arc has both of its endpoints in a curve $\gamma$ so that $\gamma \in \Gamma_2 \cup \Gamma_3$, then the arc is parallel to a subarc of $\gamma$. Let $\Delta$ be the union of all these, possibly nested, disks of parallelism.

Let $p$ be a point in $S-(\mathcal{D}\cup \Delta)$. Let $F$ be the twice punctured disk obtained by removing a neighborhood of $p$ from $S$. Then $\mathcal{D}_u \subset F$ and no curve in $\Sigma'_u\cap F$ is parallel to $\bdd F$. Furthermore the boundary of a regular neighborhood of $(\mathcal{E}\cap S)\cup\Gamma$ contains at most one curve that bounds a twice punctured disk in $F$. This curve is obtained by taking a regular neighborhood of the component of $\Gamma\cup (\mathcal{E}\cap S)$ containing $\gamma_2 $, $\gamma_3$ and an arc of $\mathcal{E}\cap S$ connecting the two. Let $\mathcal{D}_s$ be the collection of possibly punctured disks that the boundary of a regular neighborhood of $(\mathcal{E}\cap S)\cup\Gamma$ bounds in $F$.

By the claim, the boundary of a regular neighborhood of $\Sigma'_u \cup \mathcal{E}$ contains two surfaces. One is isotopic to $\Sigma'_u$ relative $S$, and the other one is isotopic to $\Sigma'_s$ relative to $S$. Then $\Sigma'_s\cap S = \bdd \mathcal{D}_s$. Let $(\Sigma'_s)^{\mathcal{D}_s}$ be the surface obtained from $\Sigma'_s$ by compressing it along $\mathcal{D}_s$. Note that some of these compressions may be along cut-disks and at most one may be along a disk with two punctures. Therefore $\chi(\Sigma'_s \cap M^-_2)\leq \chi((\Sigma'_s)^{\mathcal{D}_s}\cap M^-_2)+2$.
Note that $(\Sigma'_s)^{\mathcal{D}_s} \cap M^-_2$ is isotopic to $(\Sigma'_u)^{\mathcal{D}_u} \cap M^-_2$ so has Euler Characteristic at most $\chi(\Sigma)$ and therefore $\chi(\Sigma'_s \cap M^-_2)\leq \chi(\Sigma)+2$. On the other hand, by our choice of $s$, $\chi(\Sigma'_s \cap M^-_1)\leq 2\chi(\Sigma)$. Therefore $\chi(\Sigma')=\chi(\Sigma'_s \cap M^-_1)+\chi(\Sigma'_s \cap M^-_2)\leq 3\chi(\Sigma)+2$ as desired.

\medskip 
Case 3: There are at most two valence two or valence four vertices at the same level or there is a valence 6 vertex.  As in Theorem \ref{thm:flip} this implies that $\chi(\Sigma)\geq -3$ contrary to our hypothesis.

\end{proof}

\end{document}